\documentclass[reqno, 11pt]{amsart}
\usepackage{pgfplots}
\pgfplotsset{compat=1.17}  
\usepackage{caption}
\usetikzlibrary{intersections} 
\usepgfplotslibrary{fillbetween}

\usepackage{lmodern}
\usepackage[T1]{fontenc}

\makeatletter
\@namedef{subjclassname@2020}{%
  \textup{2020} Mathematics Subject Classification}
\makeatother

\usepackage{ulem}

\usepackage{amsfonts, amsthm, amsmath, amssymb}
\usepackage{hyperref}
\hypersetup{colorlinks=false}

\usepackage{bbm}  

 \usepackage{tabu}

\usepackage[margin=1in]{geometry}

\RequirePackage{mathrsfs} \let\mathcal\mathscr

\usepackage{hyperref}
\usepackage{dsfont}
\usepackage{upgreek} 
\usepackage{enumerate}
\usepackage{comment}

\numberwithin{equation}{section}

\newtheorem{theorem}{Theorem}[section] 
\newtheorem{lemma}[theorem]{Lemma}

\theoremstyle{definition}

\newtheorem*{acknowledgements}{Acknowledgements}
\newtheorem{remark}[theorem]{Remark}

\newtheorem*{notation}{Notation}

\renewcommand{\emph}[1]{\textit{#1}}

\renewcommand{\phi}{\varphi}

\newcommand{\0}{\mathbf{0}}

\renewcommand{\leq}{\leqslant}
\renewcommand{\le}{\leqslant}
\renewcommand{\geq}{\geqslant}

\renewcommand{\c}{\mathbf{c}}

 \renewcommand{\b}{\mathbf{b}}

\renewcommand{\k}{\mathbf{k}}

\newcommand{\md}[1]{  \left(\textnormal{mod}\ #1\right)}

\newcommand{\F}{\mathbb{F}}

\newcommand{\Z}{\mathbb{Z}} 
 
\renewcommand{\b}{\mathbf}
\renewcommand{\c}{\mathcal}
\renewcommand{\epsilon}{\varepsilon}

\renewcommand{\leq}{\leqslant}
\renewcommand{\geq}{\geqslant}
\renewcommand{\#}{\sharp}

\title
[Counting square-free values of random polynomials]
{Counting square-free values of random polynomials}

\author{Efthymios Sofos} 
\address{Universit\` a di Roma Tor Vergata\\ Dipartimento di Matematica\\00133, Rome}
\email{sofos@mat.uniroma2.it}

\date{}

\begin{document} 
\begin{abstract}We prove that the average error term when 
counting square-free values of polynomials 
is the quartic root of the main term. 
\end{abstract}

\maketitle

\setcounter{tocdepth}{1}
\tableofcontents 
 

\section{Introduction}   \label{s:intro}   
Let $P$ be an integer    polynomial.
A folklore conjecture states that 
$$ E_P(x):= \#\{n\in \Z\cap [1,x]: 
P(n) \ \textrm{square-free}\}-\mathfrak S_P x=o(x)
\ \ \ \textrm{ as } x\to \infty,$$ where 
$$ \mathfrak S_P:= \prod_{\substack{ \ell \textrm{ prime}\\\ell=2 }}^\infty
\left(1-\frac{\#\{t\in \Z/\ell^2\Z: P(t)=0\}}{\ell^2}\right).$$ The reader is referred to the 
  work of 
Browning \cite{MR2773215} 
for  a   list of 
results towards this conjecture.
It must be   noted that
even the existence 
of infinitely many square-free values 
represented by an 
irreducible polynomial is only known
when its degree is at most $3$.

Assuming    the abc-conjecture,
Granville \cite{MR1654759}
proved the conjecture, 
a work that was later 
extended by Poonen~\cite{MR1980998},
Lee--Murty \cite{MR2336969} and 
Murty--Pasten~\cite{MR3256849}.

Another direction was taken by 
Filaseta~\cite{MR1149859},
Shparlinski~\cite{MR3121693},
Browning--Shparlinski~\cite{MR4727085} 
and Jelinek~\cite{2308.15146},
who proved an averaged version of the conjecture.
Regarding the first moment,~\cite[Theorem 1.1]{MR4727085}
states that for any $A>0,d\geq 4$ and
$\alpha<\frac{1}{d-3}$ there exists 
$\delta>0$ such that  
for
$H^{1/A}<x  \leq H^{\alpha}
$ we have 
$$(2H)^{-d-1}\sum_{\substack{ 
\mathrm{primitive} \  
P\in \Z[t]\\ 
\deg(P)=d,
|P|\leq H }} | E_P(x)| = 
O( x^{1-\delta}),$$ where 
$|P|$ denotes the maximum modulus of the coefficients of $P$. 
Regarding the second  moment,~\cite[Theorem 1.4, Remark 1.7]{2308.15146}
states that for any $A>0,d\geq 2$ and
$\alpha<\frac{1}{d/3+9/5}$ there exists 
$\delta>0$ such that  
for  
$H^{1/A}<x  \leq H^{\alpha}
$ we have 
$$(2H)^{-d-1}\sum_{\substack{ 
\mathrm{primitive} \  
P\in \Z[t]\\ 
\deg(P)=d,
|P|\leq H }} | E_P(x)|^2 = 
O( x^{2-\delta}).$$

Let  $\zeta$ denote 
the Riemann zeta function and define  
\begin{equation}\label{Leo_Artemisia}\gamma_0:= -4\frac{\zeta(-1/2) }{\zeta(2)^2}
\prod_{\substack{\ell \ \mathrm{ prime} \\ \ell=2}}^\infty
\left(1-\frac
{(2  - \ell^{-1/2}+ \ell^{-1})}
{(\ell + 1)^2 }
\right)=0.192\ldots \ .\end{equation}
We show that $E_P(x)$ is typically $x^{1/4}$ by 
proving asymptotics for the second moment: 
\begin{theorem}\label{mmm:ttt} Let $d,x$ be integers 
strictly larger than $2$. 
\begin{enumerate}\item 
For any $\alpha<\frac{1}{d+3}$ there exists 
$\delta>0$ such that  
whenever  $H$ satisfies   
$x  \leq H^{\alpha}
$ then 
$$(2H)^{-d-1}\sum_{\substack{ P\in \Z[t]\\
\deg(P)=d, \ 
|P|\leq H }} | E_P(x)|^2 = \gamma_0 x^{1/2} + 
O( x^{1/2-\delta})  $$ holds 
with an implied constant depending only on 
$\alpha$ and $\delta$.
\item 
For any $\epsilon>0, 0<\beta<\frac{1}{d+28/9}$ 
and all $H$ satisfying
$x  \leq H^{\beta}
$ we have $$(2H)^{-d-1}\sum_{\substack{ P\in \Z[t]\\
\deg(P)=d,\ 
|P|\leq H }} | E_P(x)|^2 = \gamma_0 x^{1/2} + 
O(  x^{4/9+\epsilon} )  ,$$ where 
the implied constants 
depend only on $\epsilon$ and $\beta$. 
\end{enumerate}
\end{theorem} 
The order of magnitude $x^{1/2}$ is not  
apparent at first glance
since the terms $n=m$ in
$$\sum_{\substack{ P\in \Z[t]\\
\deg(P)=d,\ 
|P|\leq H }}\left(\sum_{n\leq x}\mu(n)^2\right)^2 =\sum_{\substack{ P\in \Z[t]\\
\deg(P)=d,\ 
|P|\leq H }} \sum_{n,m\leq x}\mu(n)^2 \mu(m)^2 $$
  contribute  
$c x H^{d+1} $ for some constant $c$, where 
 $\mu$ denotes the M\"obius funtion.
We will prove a full asymptotic expansion 
for the contribution of the 
 terms $n\neq m$ that contains the expression
$-c H^{d+1} x$. 
 The new ingredient    is the recognition 
that  Ces\`aro summation is built in
second moments of arithmetic functions over random polynomials of fixed degree. This allows us to use 
Perron type integrals with extra convergent factors,   
utilize bounds for the zeta function on the critical strip and shift the line of integration   the left of $\Re(s)=-1/2$. As will be explained in Remark \ref{rem:nolindelof}
the Lindel\"of hypothesis or  
subconvexity bounds 
do not offer any advantage.
\begin{remark}[Higher moments]
When one replaces $\mu^2$ 
by the Liouville $\lambda$ function
and   $x\leq H^{O(1)}$ by $x\leq (\log H)^{O(1)}$,
the first $d+1$ moments were 
asymptotically estimated by Wilson~\cite{MR4894179}, whereas, 
all moments were successively asymptotically estimated by Kravitz--Woo--Xu \cite{2512.03292}.
Given the work of Mirsky \cite{MR21566}, 
asymptotics for all moments  for $\mu^2$ are possible, however, computations appear forbidding, especially regarding Ces\`aro summation. It would be interesting to see if 
$E_P(x)/\sqrt x$ follows a probability distribution;
for $\lambda$ the distribution   is Gaussian \cite{2512.03292}.
\end{remark}

\begin{notation} For an integer polynomial $P$ we 
denote the maximum modulus of its coefficients by $$|P|.$$  For  a non-zero integer $n$ we write $$n\mid P$$ to mean that $n$ divides all coefficients of $P$. We denote the discriminant of $P$ by  $D_P$.  Throughout the paper $$d$$ will denote a fixed positive integer and we will use the notation 
$$\c P_d(H):=\{P\in \Z[t]: 
\deg(P)=d, |P|\leq H\}.$$ \end{notation}

\begin{acknowledgements} 
I thank Igor Shparlinski for helpful comments that improved the clarity of the presentation.  
\end{acknowledgements}

\section{Preparations} We shall work with the truncated version 
$$\mathfrak S_P(z)=\sum_{1\leq k \leq z} \mu(k)  \frac{\varrho_P(k^2) }{k^2}$$ 
where $z\geq 1$ is a function of $d,x,H$
that will be determined later. 
\begin{lemma} \label{lem:boundsforprimes} For $P\in \Z[t]$  and 
for any prime $\ell$ we have  
$$ \varrho_P(\ell^2) \leq  \begin{cases}
\ell^2,&\ell^2\mid P, \\ 
\ell \deg(P),&\ell\mid P, \ell^2\nmid P, \\ 
\deg(P),& \ell\nmid D_P, \\ 
\deg(P) (\ell+1),& \ell\mid D_P,\ell\nmid P.
 \end{cases}$$\end{lemma}\begin{proof}The first case is trivial. Assume   that $\ell^2\nmid P$.
 If $\ell\mid P$ then $ \varrho_P(\ell^2) 
=\ell\varrho_{P/\ell}(\ell)$, which is at most 
$\ell \deg(P)$ because $P/\ell$ is not identically $0$ in $\F_\ell$. This settles the second case
and we   assume that $\ell\nmid P$ for the rest of the proof. Write $ \varrho_P(\ell^2) = \sum_{t} N(t)$, where the 
 sum is over $t\in \F_\ell $ with $P(t)=0$ and $N(t)$ is the number of $x\in \Z/{\ell^2}\Z$ with  $P(x)=0$ and $ x\equiv t \md \ell$. 
 If $P'(t) \neq 0$ then $N(t) \leq 1 $ and if 
 $P'(t) = 0$ then $N(t) \leq \ell$. We obtain 
   $$ \varrho_P(\ell^2) \leq 
  \#\{  t\in \F_\ell:  P(t)=0,P'(t) \neq 0 \} +
  \ell \#\{  t\in \F_\ell:  P(t)=0=P'(t) \},$$ which is $\leq \deg(P)
  +\ell \mathds 1_{\ell \mid D_P}
  \#\{  t\in \F_\ell:  P(t)=0=P'(t) \}$
 because  $\ell\nmid P$. This is at most $\deg(P)$
 when $\ell\nmid D_P$, thus, settling the third case.
 If $\ell \mid D_P$, then we 
 use $  \#\{  t\in \F_\ell:  P(t)=0=P'(t) \}\leq \deg(P)$
 to conclude the proof. 
 \end{proof}

 \begin{lemma} \label{lem:unweightedsum} For 
 $P\in \Z[t]$ with $\deg(P)\geq 1$,
 any $\epsilon>0$ and $x\geq  1 $ 
 we have  $$ \sum_{1\leq k \leq x} \mu(k)^2 \varrho_P(k^2) \ll_{\deg(P),\epsilon} (3\deg(P))^{2\omega(D_P)} x^{1+\epsilon}  
 \sum_{k\in \mathbb N, k^2\mid P } k,$$
where the implied constant depends only on 
$\deg(P)$ and $\epsilon$. If $D_P=0$ then 
$$\sum_{1\leq k \leq x} \mu(k)^2 \varrho_P(k^2) 
\ll_{\deg(P),\epsilon} x^\epsilon
\c C(P)^{1+\epsilon}
,$$ where $\c C(P)$ is the content of $P$.  
\end{lemma}\begin{proof} Assume that $D_P\neq 0$ and write $k=k_1k_1k_2k_3$, where   $k_i$ is the product of all primes $\ell\mid k$ corresponding to the $i$-th case 
of Lemma \ref{lem:boundsforprimes}. 
Then the sum over $k$ in the lemma is at most 
$$   
 \sum_{k_1^2\mid P }   k_1^2
 \sum_{  k_2\mid P }  \mu(k_2  )^2 k_2 
 \deg(P)^{\omega(k_2)}
 \sum_{  k_4  \mid D_P   }  \mu(k_4 )^2      (\deg(P))^{\omega(k_4)}   \sigma(k_4)   
 \sum_{ \substack{ k_3  \leq x/ (k_1 k_2 k_4) 
 }} \mu(k_3)^2  {\deg(P)}^{\omega(k_3)}     
 ,$$ where $\sigma$ is the sum of divisors function.  The sum over $k_3$ is $O_{\deg(P),\epsilon}( x^{1+\epsilon}/(k_1k_2k_4)  )$, thus, we get 
 $$ \ll_{\deg(P),\epsilon}  x^{1+\epsilon} 
  \sum_{k_1^2\mid P }   k_1
 \sum_{   k_2 \mid  P }    
 \mu(k_2 )^2   {\deg(P)}^{\omega(k_2)}   
 \sum_{k_4\mid D_P}
 \mu(k_4)^2  \frac{ \sigma(k_4)   }{k_4} (\deg(P))^{\omega(k_4)}  .$$
 The sum over $k_2$ is at most $ (2\deg(P))^{\omega(P)}
 \leq (2\deg(P))^{\omega(D_P)}$
 and the sum over $k_4$ is $$
 \leq 
\prod_{\ell\mid D_P} \left( 1+\frac{\deg(P)(\ell+1) }{\ell}\right)\leq (3\deg(P))^{\omega(D_P)}.$$ 
This completes the proof when $D_P\neq 0$.

If $D_P=0$ then we note that 
$\rho_P(\ell^2)\leq \ell^2$ when $\ell^2\mid P$
and $\rho_P(\ell^2)\leq \deg(P)(\ell+1)$ when $\ell^2\nmid P$. Factoring $k$ as $k_1k_2$,
where $k_1$ is composed of primes $\ell\mid k$ with 
$\ell^2\mid P$, gives the following bound for the 
sum in the lemma:
$$\sum_{ \substack{ k_1k_2\leq x\\ k_1^2\mid P, k_2\mid P }}   
\mu(k_1k_2)^2  k_1^2  \sigma(k_2) \deg(P)^{\omega(k_2)}
\ll_{\deg(P),\epsilon} x^\epsilon
\sum_{ \substack{  k_1^2\mid P\\ k_2\mid P }}  \mu(k_1k_2)^2    k_1^2  k_2  
.$$  Note that $k_1^2k_2$ is upper-bounded by the content of $P$, hence, the divisor bound completes the proof. 
\end{proof}

\begin{lemma}\label{lem:boundtailsquare} Fix $d\geq 3$ and 
$ \epsilon>0$. Then for all $H,z\geq 1$ 
we have  $$ 
\sum_{\substack{ P\in 
\c P_d(H)\\
 D_P\neq 0 }} \left(  \sum_{k>z} \mu(k)  \frac{ \varrho_P(k^2) }{ k^2 } \right)^2\ll_{d,\epsilon}   \frac{H^{d+1+\epsilon}}{z^{2-\epsilon}}
,$$ where the implied constant 
depends only on $\deg(P)$ and $\epsilon$.
\end{lemma}\begin{proof}When $P\in \c P_d(H)$ and $D_P\neq 0$,
Lemma \ref{lem:unweightedsum}
and   partial summation yields \begin{equation}\label{ifigenia}
\left|  \sum_{t>z} \mu(t)  
\frac{ \varrho_P(t^2) }{ t^2 } \right| \ll_{d,\epsilon} \frac{  (3d)^{\omega(D_P)}  }
{z^{1-\epsilon}}    \sum_{k^2\mid P }  k
\ll_{d,\epsilon}\frac{   H^\epsilon  }
{z^{1-\epsilon}}    \sum_{k^2\mid P} \mu(k)^2  k\end{equation} uniformly in $P$ and $z$. 
This is due to $|P|\leq H$,
the bound $D_P\ll H^{c(d)}$ that holds for some $c(d)\geq 0$ and    
$\omega(t)\ll (\log t)/(\log \log t)$. 
The overall bound then becomes 
$$ \ll \frac{(Hz)^{2\epsilon}}{z^2}
\sum_{\substack{ P\in \c P_d(H)
\\ D_P\neq 0}} \left(\sum_{k^2\mid P}k\right)^2= \frac{(Hz)^{2\epsilon}}{z^2}
\sum_{k_1,k_2 \in \mathbb N}k_1k_2
\sharp\{ P\in \c P_d(H):
[k_1,k_2]^2\mid P \},$$ where $[k_1,k_2]$ is the least common multiple. Let $\delta:=\gcd(k_1,k_2)$ so that 
$k_i=\delta s_i$ and the sum over $k_i$ becomes 
\[\ll 
\sum_{k_1,k_2 \in \mathbb N} k_1k_2
\left(\frac{H}{[k_1,k_2]^2}
\right)^{d+1}
\leq 
\sum_{\delta, s_1,s_2 \in \mathbb N} \delta^2 s_1s_2
\left(\frac{H}{(\delta s_1s_2)^2}
\right)^{d+1}
\ll H^{d+1}.\] When $D_P=0$,
we use Lemma \ref{lem:unweightedsum} to get 
$$\left|  \sum_{t>z} \mu(t)  \frac{ \varrho_P(t^2) }{ t^2 } \right| \ll_{d,\epsilon}  
\frac{\c C(P)^{1+\epsilon}}{z^{1-\epsilon}}
,$$ hence, $$
\sum_{\substack{ P\in \c P_d(H)\\
D_P= 0  }} \left(  \sum_{k>z} \mu(k)  \frac{ \varrho_P(k^2) }{ k^2 } \right)^2
\ll_{d,\epsilon} 
z^{-2+\epsilon} H^\epsilon
\sum_{\substack{ P\in \c P_d(H) \\  D_P= 0 }}  \c C(P)^{2} . $$ Letting $\c V:=\c C(P)$
and writing $Q=P/\c V$,
the sum in the right-hand side becomes 
$$\sum_{\c V\leq H } \c V^2 
\sum_{\substack{ Q\in \c P_d(H/\c V)
\\D_Q= 0  }}  1\ll  
\sum_{\c V\leq H } \c V^2 \frac{H^d}{\c V^d}
\ll H^d \log H$$ by   \cite{MR2505814} 
and $d\geq 3 $.  \end{proof}

\section{Transition to the model}
\begin{lemma}\label{lab:lemafirstmom}
For any $x,H,z,d\geq 1$ we have 
$$ (2H)^{-d-1}
\sum_{\substack{ P\in \c P_d(H)}} 
\sum_{n\leq x} \mu(P(n ))^2
=\frac{x}{\zeta(2)}   +O\left(
\frac{  x}{z}
+ \frac{x^{1+d/2}}{H^{1/2}}+
  \frac{x z}{H} \right),$$ where 
  the implied constant depends only
on $d$.
\end{lemma}
\begin{proof}
We have 
$$\sum_{P\in \c P_d(H)} 
\sum_{n\leq x} \mu(P(n ))^2
=\sum_{n\leq x} \sum_{k\leq z} \mu(k) 
\#\{P\in \c P_d(H): k^2 \mid P(n )\}
+O(E),$$ where 
the implied constant is absolute and 
$$E= \sum_{n\leq x} 
\sum_{z<k\ll \sqrt{ (d+1) x^d H}} 
|\mu(k)| \#\{P\in \c P_d(H): k^2 \mid P(n )\}
.$$ Fixing all coefficients of $P$ except
the constant 
we obtain $$ \#\{P\in \c P_d(H): 
k^2 \mid P(n )\}\ll
H^{d} \left( \frac{H}{k^2} +1 \right)
\ll \frac{H^{d+1}}{k^2}+H^d ,$$ thus, 
$E\ll x (H^{d+1}z^{-1}+ H^{d+1/2} x^{d/2} ) $.
Writing $P=\sum_{j=0}^d c_j t^j$ 
the main term equals   $$ 
\sum_{\substack{n\leq x\\k\leq z}} \mu(k) 
\sum_{|c_1,\ldots, c_d|\leq H}
\#\left\{|c_0|\leq H: c_0\equiv -\sum_{j=1 }^d c_j n^j \md{k^2}\right\}
.$$ The inner cardinality equals $2H k^{-2}+O(1)$, hence, we get 
\[\sum_{\substack{n\leq x\\k\leq z}} \mu(k) 
((2H)^{d+1}k^{-2}+O(H^{d}))
= \frac{x}{\zeta(2)} (2H)^{d+1} +O\left( \frac{x H^{d+1} }{z}+ xH^d z \right).\qedhere\]
\end{proof}
 For $x,H,z\geq 1 $ write  
$$ \mathfrak S_P(z):=\sum_{k\leq z} \mu(k) \frac{\varrho_P(k^2)}{k^2}\ \ \ \textrm{ and } \ \ \ 
V(H,x,z):=\sum_{P\in \c P_d(H)}
\left( S_P(x) - \mathfrak S_P(z) x\right)^2.$$ 
\begin{lemma}\label{lem:initialsplit} 
For any $H,z,x\geq 1$ and any fixed $\epsilon>0$
we have $$ \sum_{P\in \c P_d(H)}
  E_P(x)^2 =V(H,x,z)+O_{d,\epsilon}\left(V(H,x,z)^{1/2} x
\left( \frac{H^{d+1+\epsilon}}
{z^{2-\epsilon}} \right)^{1/2}
+  x ^2  \frac{H^{d+1+\epsilon}}{z^{2-\epsilon}} \right),$$   with an implied constant that depends only on 
$\epsilon$ and $d$.\end{lemma}
\begin{proof}  
Since $E_P(x)=O(x)$, the sum in the lemma equals 
\begin{equation}\label{eq:kuba}\sum_{P\in \c P_d(H)}
( S_P(x) - \mathfrak S_P x)^2 +O(x^2 H^d \log H)
.\end{equation}We expand $( S_P(x) - \mathfrak S_P x)^2 $ as  $$
( S_P(x) - \mathfrak S_P(z) x)^2+O(
| S_P(x) - \mathfrak S_P(z) x|\cdot |\mathfrak S_P-\mathfrak S_P(z)| x
+|\mathfrak S_P-\mathfrak S_P(z)|^2 x ^2).$$ 
Cauchy's inequality 
shows that the sum in \eqref{eq:kuba}
equals
$V(H,x,z)$ up to a quantity of modulus   
$$\ll
V(H,x,z)^{1/2} x 
\left(\sum_{P\in \c P_d(H)}|\mathfrak S_P-\mathfrak S_P(z)|^2 \right)^{1/2}
+  x ^2\sum_{P\in \c P_d(H)}
|\mathfrak S_P-\mathfrak S_P(z)|^2.$$ Alluding to
   Lemma \ref{lem:boundtailsquare} concludes the proof.
 \end{proof} 
For   $x,z\geq 1$, $n \in \mathbb Z, P\in \Z[t]$ define $$
\mu^2_{\mathrm{model}}(n):=\sum_{\substack{k\in \mathbb N
\cap [1,z]\\ k^2 \mid n }}\mu(k)
\ \ \ \textrm{ and } \ \ \   
S_P(x,z)=\sum_{1\leq n \leq x } \mu^2_{\mathrm{model}}(P(n))
.$$\begin{lemma}\label{lem:cutdown}
For any 
$x,z,H\geq 1 $ we have 
\begin{align*}V(H,x,z)&=
(2H)^{d+1} \frac{x}{\zeta(2)}
 +O\left( x z H^d
+  (xH)^\epsilon\left(\frac{x^2 H^{d+1}}{z}+
x^{2+d/2} H^{d+1/2}\right)\right)\\
&+
\sum_{1\leq n\neq m \leq x } 
\sum_{P\in \c P_d(H)}
\mu^2_{\mathrm{model}}(P(n))
\mu^2_{\mathrm{model}}(P(m))
\\&-2x
\sum_{1\leq n  \leq x } 
\sum_{P\in \c P_d(H)}
\mu^2_{\mathrm{model}}(P(n))
\mathfrak S_P(z)  
+x^2 \sum_{1\leq n  \leq x } 
\sum_{P\in \c P_d(H)} \mathfrak S_P(z)^2
.\end{align*}
\end{lemma}\begin{proof}
The quantity $V(H,x,z)$ equals 
 \begin{equation}
    \label{dispersion}\sum_{n,m \leq x}
\sum_{ P\in \c P_d(H)} 
\mu(P(n ))^2 \mu(P(m ))^2
-2x\sum_{n \leq x}
\sum_{P\in \c P_d(H)}   
\mathfrak S_P(z)  \mu(P(n ))^2
+x^2 
\sum_{P\in \c P_d(H)}   
\mathfrak S_P(z)^2 .
\end{equation} 
The terms $n=m$  give rise to $\sum_{n\leq x}
\sum_P \mu(P(n))^2$  
that  can be estimated by   Lemma \ref{lab:lemafirstmom}.
Writing  $\mu(P(m))^2$ as 
$\sum_{k^2\mid P(m)} \mu(k)^2$,
we see that the contribution of the terms   $k>z$ towards \eqref{dispersion} is 
$$\ll \sum_{n\neq m \leq x} 
\sum_{ P\in \c P_d(H)}    \mu(P(n))^2
\sum_{\substack{k>z\\k\mid P(m)}}  \mu(k) ^2
\ll x E,$$ where $E$ is as in the proof of 
Lemma \ref{lab:lemafirstmom}. Using the bound 
for $E$ proved there we obtain the  
error term $\ll  x^2 (H^{d+1}z^{-1}+ H^{d+1/2} x^{d/2} ) .$ Thus, up to an admissible 
error term, the  
first term in  \eqref{dispersion} is 
$$\sum_{n\neq m \leq x}
\sum_{ P\in \c P_d(H), \ D_P\neq 0}     \mu(P(n ))^2 
\mu^2_{\mathrm{model}}(P(m ))
.$$ We next bound 
the contribution of the terms $k>z$ 
in
$ \mu(P(n ))^2=\sum_{k\mid P(n))} \mu(k)^2$
similarly. By  the   bound $\mu^2_{\mathrm{model}}(P(m ))
\ll \tau(P(m))\ll_\epsilon (xH)^\epsilon$  we obtain
$$\sum_{\substack{n\neq m \leq x\\
P\in \c P_d(H)
}}  
\left( \mu(P(n ))^2 \mu(P(m ))^2
-\mu^2_{\mathrm{model}}(P(n )) 
\mu^2_{\mathrm{model}}(P(m )) \right)
\ll (xH)^\epsilon\left(\frac{x^2 H^{d+1}}{z}+
x^{2+d/2} H^{d+1/2}\right).$$ To deal with the second term in   \eqref{dispersion} we note that   
$$\sum_{n \leq x}\sum_{   P\in \c P_d(H)
}  
\mathfrak S_P(z)
( \mu(P(n ))^2-\mu^2_{\mathrm{model}}(P(n )) ) \ll   \sum_{n \leq x}\sum_{ 
P\in \c P_d(H)}  |\mathfrak S_P(z)|
\sum_{\substack{k>z\\ k^2\mid P(n)}}\mu(k)^2
.$$ Using  \eqref{ifigenia} with $T=1$ this is 
$$ 
\ll 
x\sum_{n \leq x}  
\sum_{\substack{t\leq \sqrt H \\ z<k\ll\sqrt{Hx^d} }}
\mu(t)^2 t \mu(k)^2  \#\{P\in \c P_d(H): t^2\mid P, k^2\mid P(n)\} .$$
Writing $g:=P/t^2$ we obtain $$
\ll  \sum_{\substack{n \leq x\\ t\geq 1, k>z }}
\hspace{-0.3cm}
\mu(t)^2 t \mu(k)^2  
\#\{g\in \c P_d(H/t^2): 
(k/\gcd(k,t))^2\mid g(n)\}   
\ll \hspace{-0.5cm} 
\sum_{\substack{t\leq \sqrt H \\ z<k\ll\sqrt{Hx^d} }}
 \hspace{-0.5cm}  xt
 \left(\frac{H}{t^2}\right)^d
  \hspace{-0.1cm} 
\left(\frac{H\gcd(k,t)^2}{t^2k^2}+1\right).$$
Using $\gcd(k,t)\leq t$ this becomes \[
\ll  x H^{d+1}
\sum_{\substack{t\leq \sqrt H \\ z<k\ll\sqrt{Hx^d} }}
  \frac{1}{t^{2d-1}}
 \frac{1}{k^2}+x
\sum_{\substack{t\leq \sqrt H \\ z<k\ll\sqrt{Hx^d} }}
 \frac{H^d}{t^{2d-1}}  \ll  \frac{x H^{d+1} (\log H)}{z}+ x^{1+d/2} H^{d+1/2}\log H.\qedhere\]  
\end{proof}

\begin{lemma}
\label{lem:local}
For any prime $\ell$,
and $d\geq 1$ and $n,m\in \mathbb Z$ 
we have 
$$\frac{\#\{g\in (\Z/\ell^2\Z)[t]: \deg(g)\leq d, g(n)=g(m)=0\}}{\ell^{2(d+1)}}
=\begin{cases}\ell^{-4}, &  v_\ell(m-n)=0, \\
\ell^{-3}, & v_\ell(m-n)=1, \\
\ell^{-2}, & v_\ell(m-n)\geq 2.
\end{cases}$$\end{lemma}
\begin{proof}We write $g=\sum_{i=0}^d c_i t^i$
where $\b c \in (\Z/\ell^2\Z)^{d+1}$. Thus, 
describes a system of two linear equations in $d+1$ variables over $\Z/\ell^2\Z$. We fix $c_i$ for each $i\geq 2$
to end up with a $2\times 2$ system whose determinant is $m-n$. When $\ell\nmid m-n$ the variables $c_0$ and $c_1$ are fully determined by the remaining $c_i$, which proves the first case.
When  $v_\ell(n-m)=1$ then the system 
$g(n)=g(m)=0 \mod \ell^2$ is equivalent to 
$g(n)=0 \md {\ell^2}, g'(m)=0 \md \ell$.
The second equation determines $c_1\md \ell$ and the first equation determines $c_0\md{\ell^2}$, 
thus, we obtain $\ell^{-3}$. When $\ell^2\mid m-n$ 
the  system $g(n)=g(m)=0 \mod \ell^2$ is the 
same as $g(n)=0\md \ell^2$. In this case 
$c_0\md {\ell^2}$ is determined uniquely the remaining $c_i$, which gives the factor $\ell^{-2}$.\end{proof}

\begin{lemma}
\label{lem:local2}For any
$d\geq 1$, any prime $\ell$,
and $n\in \mathbb Z$ 
we have 
$$\sum_{\substack{ g\in (\Z/\ell^2\Z)[t], \deg(g)\leq d  \\ g(n)\equiv 0 \md{\ell^2} } }\varrho_g(\ell^2)
=\ell^{2d}\left(3-\frac{2}{\ell}\right)
\ \ \mathrm{ and } \ \ 
\sum_{\substack{ g\in (\Z/\ell^2\Z)[t]\\ \deg(g)\leq d   } }\varrho_g(\ell^2)^2
=\ell^{2d+2}\left(3-\frac{2}{\ell}\right).$$
\end{lemma}
\begin{proof}
The first sum can be written as 
 $$ \sum_{x\in \Z/\ell^2\Z}\#\{g\in (\Z/\ell^2\Z)[t]: \deg(g)\leq d, g(n)=0=g(x)\}=  \ell^2 \ell^{2(d+1)}
 \left(  \frac{\ell^2-\ell}{\ell^4}
 +\frac{ \ell-1}{\ell^3} +\frac{1}{\ell^2}
 \right)$$ by Lemma~\ref{lem:local}. This proves the first claim;
the second   follows   from the first.
\end{proof}\begin{lemma}\label{lem:middle}For any $\epsilon>0,H,z\geq 1, x\in \mathbb N$ we have 
\begin{align*}\sum_{n \leq x}
\sum_{ P\in \c P_d(H)}
\mu^2_{\mathrm{model}}(P(n))\,\mathfrak S_P(z)
&= \gamma (2H)^{d+1} x+ O\!\left(x z^\epsilon\bigl(z^{4d+2}
+H^d z^2+H^{d+1}/z\bigr)\right), \\[0.6em]
\sum_{n \le x} \sum_{P\in \c P_d(H)}
\mathfrak S_P(z)^2 &= \gamma (2H)^{d+1} 
+ O\!\left(z^{\epsilon+4d+2}+z^{\epsilon+2} H^d+H^{d+1}/z\right), \end{align*} where the implied constant depends only on $d,\epsilon$ and  we let    
$$\gamma:=\prod_{\substack{\ell \ \mathrm{ prime} \\ \ell=2}}^\infty\left( 1-\frac{2}{\ell^2}+
\frac{3-2/\ell}{\ell^3}
\right).$$ \end{lemma}\begin{proof} 
Writing  $t_0=\gcd(k_1,k_2)$  makes 
 the  first sum equal to $$\sum_{\substack{ 
n \leq x\\ k_1,  k_2\leq z }} \mu(k_1) 
\frac{\mu(k_2)}{k_2^2}
\sum_{\substack{ P\in \c P_d(H) \\ 
k_1^2\mid P(n)}} \varrho_P(k_2^2)=
\hspace{-0.2cm}
\sum_{\substack{ n \leq x, t_0t_1\leq z \\ \gcd(t_1,t_2)=1, t_0t_2\leq T }} 
\hspace{-0.5cm}
\mu(t_0t_1) \frac{\mu(t_0t_2)}{(t_0t_2)^2}
\sum_{\substack{ g\md{(t_0 t_1 t_2 )^2}\\ 
\deg(g)\leq d,
(t_0 t_1)^2\mid g(n)}} \varrho_P((t_0 t_2)^2)
\Xi,$$
where  $\Xi$
is the cardinality of $P\in \c P_d(H)$ with 
$P\equiv g 
\md{(t_0 t_1 t_2)^2}$. We       estimate 
$\Xi$ as \newline
$ (2H)^{d+1}(t_0 t_1 t_2 )^{-2d-2} 
+O(1+H^{d}(t_0 t_1 t_2 )^{-2d} )$ and use Lemma \ref{lem:local2} to bound the error term 
by 
$$\ll x \sum_{\substack{   t_0t_1\leq z \\  t_0t_2\leq z }} 
\frac{ \mu^2(t_0t_1t_2)}{(t_0t_2)^2}
(1+H^{d}(t_0 t_1 t_2 )^{-2d} )
 t_1^{2d}t_2^{2d+2}
t_0^{2d}3^{\omega(t_0)}
\ll xz^\epsilon(z^{4d+2}+H^d z^2).$$
By Lemma \ref{lem:local2} the main term contribution 
is
$$ 
\sum_{\substack{  \max\{ t_0t_1,t_0t_2\}\leq z \\ \gcd(t_1,t_2)=1  }}  
\mu(t_0t_1) \frac{\mu(t_0t_2)}{(t_0^2t_1t_2)^2}
\prod_{\ell\mid t_0 } \left(3-\frac{2}
{\ell}\right)
=\gamma+O_\epsilon(z^{\epsilon-1}).$$
Following a similar   path   we can write 
the second sum in the lemma 
as $$ \sum_{\substack{t_0t_1\leq z\\ 
\gcd(t_0,t_1)=1,  t_0 t_2 \leq z}} 
 \frac{\mu( t_0 t_1 ) \mu( t_0 t_2)  }
 {(t_0^2t_1t_2)^2} 
 \sum_{\substack{g\md{(t_0 t_1 t_2 )^2 }\\\deg(g)\leq d } } 
 \varrho_P((t_0 t_1)^2) \varrho_P((t_0 t_2)^2) \Xi  .$$ 
By Lemma \ref{lem:local2} the error term is 
$$\ll z^\epsilon
\sum_{\substack{t_0t_1\leq z\\ 
  t_0 t_2 \leq z}} 
(t_0 t_1 t_2 )^{2d} 
(1+H^{d}(t_0 t_1 t_2 )^{-2d} )
\ll z^{\epsilon+4d+2}+z^{\epsilon+2} H^d.$$
 The main term  can be dealth with analogously.
 \end{proof}\begin{lemma}\label
 {lem:morte d'abel avondano}For any $\epsilon>0,H,z,x\geq 1$ and $n\neq m \in \mathbb N$  we have 
$$\sum_{P\in \c P_d(H)}
\mu^2_{\mathrm{model}}(P(n))
\mu^2_{\mathrm{model}}(P(m))
=(2H)^{d+1}c(n,m)
+ O\left(
z^{4d+2}+z^2H^d+
\frac{H^{d+1} \log z}{z}
\right), 
$$ where the implied constant depends only on $d,\epsilon$ and  we let    
$$c(m,n):=\prod_{\substack{\ell \ \mathrm{ prime} \\ \ell=2}}^\infty
\left(1-\frac{2}{\ell^2} +
\frac{1}{\ell^2} 
\begin{cases}
\ell^{-2}, & \ell \nmid m-n, \\
\ell^{-1}, & v_\ell(m-n)=1, \\
1, & \ell^2 \mid m-n
\end{cases}
\right)  .$$ \end{lemma}\begin{proof} 
We write the sum as 
$$ \sum_{k_1,k_2\leq z} \mu(k_1)\mu(k_2)
\sum_{\substack{ g \md{(k_1 k_2)^2 }, \deg(g)\leq d  \\ k_1^2 \mid g(n), k_2^2 \mid g(m) }}\sum_{\substack{ P\in \c P_d(H)\\ P\equiv g \md{(k_1k_2)^2} }}1.$$ Factoring 
$k_i$ as $t_0 t_i$ and recalling 
the definition of $\Xi$  in the proof of 
Lemma \ref{lem:middle} we get  $$ \sum_{\substack{ t_0 t_1 \leq z \\
\gcd(t_1,t_2)=1, t_0 t_2 \leq z }}
 \mu( t_0 t_1 )\mu( t_0 t_2)
\sum_{\substack{ g \md{(k_1 k_2)^2 }, \deg(g)\leq d  \\ k_1^2 \mid g(n), k_2^2 \mid g(m) }}\Xi,$$
 By Lemma \ref{lem:local}
the number of $g$ is $ (t_0 t_1 t_2)^{2d}f_{m-n}(t_0)$, where 
$$ f_k(t):=\prod_{\ell\mid t }  \begin{cases}\ell^{-2}, & \text{if } v_\ell(k)=0, \\
\ell^{-1}, & \text{if } v_\ell(k)=1, \\
1, & \text{if } v_\ell(k)\geq 2.
\end{cases} $$ Now the estimate 
$\Xi=(2H)^{d+1}(t_0 t_1 t_2 )^{-2d-2} 
+O(1+H^{d}(t_0 t_1 t_2 )^{-2d} )$
leads to the expression
$$ \sum_{\substack{ t_0 t_1, t_0 t_2 \leq z \\
\gcd(t_1,t_2)=1 }} \mu( t_0 t_1 )\mu( t_0 t_2)
 f_{m-n}(t_0)
\left(
\frac{(2H)^{d+1}}{(t_0 t_1 t_2 )^2 }
+O( (t_0 t_1 t_2)^{2d}+H^{d}  )
\right)
.$$
The error term contribution is 
$$
\ll 
\sum_{t_0\leq z }
\sum_{t_1,t_2\leq z/t_0}
( (t_0 t_1 t_2)^{2d}+H^{d}  )
\ll z^{4d+2}+z^2H^d
.$$ The main term equals 
$$ (2H)^{d+1}
\sum_{\substack{ t_0 t_1, t_0 t_2 \leq z \\
\gcd(t_1,t_2)=1 }} 
\frac{\mu( t_0 t_1 )\mu( t_0 t_2)}
{(t_0 t_1 t_2 )^2 }  f_{m-n}(t_0).$$
Removing the size restrictions on $t_0 t_i$ 
can be done at the cost of 
$$\ll H^{d+1} \sum_{ t_0 t_1>z} 
\frac{1}{(t_0 t_1)^2 } =
H^{d+1} \sum_{n>z} \frac{\tau(n)}{n^2}\ll\frac{
H^{d+1} \log z}{z}.$$
We end up with the main term $ (2H)^{d+1}$ multiplied by the constant  \[ 
\sum_{\substack{ \b t \in \mathbb N^3 \\
\gcd(t_1,t_2)=1 }} 
\frac{\mu( t_0 t_1 )\mu( t_0 t_2)}
{(t_0 t_1 t_2 )^2 }  f_{m-n}(t_0) 
=c(m,n).\qedhere\] \end{proof}

\section{Ces\`aro summation}
Even if a sum   $S(N)=\sum_{n\leq N} b_n$ diverges as $N\to\infty$,
the average  of the sums
$$\frac{1}{x}\sum_{N\leq x} S(N)=
\frac{1}{x}\sum_{n\leq x} b_n (x+1-n), \ \ \ x\in \mathbb N$$ might still converge. If
 not,  one can iterate this process 
to increase the chance of convergence.
This procedure, typically called Ces\`aro summation,
is common in Fourier analysis.

In our set-up, Ces\`aro summation
is built in the second moment expression.
Indeed, the non-diagonal terms in the second moment
give rise to 
$$\sum_{1\leq n\neq m\leq x }\sum_{\substack{|P|\leq H}}
\mu^2_{\mathrm{model}}(P(n))
\mu^2_{\mathrm{model}}(P(m)),
$$ which, by Lemma \ref{lem:morte d'abel avondano}, 
is approximated by  
$$(2H)^{d+1}\sum_{1\leq n\neq m\leq x } c(m,n)
.$$
By definition, $c$ is     a function of $m-n$
and letting $k=m-n$ we get $$(2H)^{d+1}
2\sum_{1\leq k\leq x } (x-k) b(k)
, \ \ \ x\in \mathbb N,$$ where $b(m-n)=c(m,n)$.
This allows us to use a   version of 
Perron's formula 
from~\cite[\S 5.1.1]{MR2378655}:
$$\sum_{1\leq n\neq m\leq x } c(m,n)
= 
\frac{1}{\pi i } \int_{\Re(s)=\theta}
\left(\sum_{k=1}^\infty \frac{b(k)}{k^s}\right)
\frac{x^{s+1}}{s(s+1)} \mathrm ds$$ for all $\theta>1$. This converges owing to 
$|c(m,n)|=\prod_{\ell} (1+O(\ell^{-2}))=O(1)$,
uniformly in $m,n$.  
\begin{remark}\label{rem:integration}
The situation is delicate since 
 we are aiming at an error term of size $O(\sqrt x)$. This can be achieved only if 
 we   shift the contour of integration into a subset of  $\Re(s)\leq -1/2$.
It   turns out that the Dirichlet series of $b(k)$
is essentially $\zeta(s)$ and this causes   
difficulties as the functional equation shows that the size of $\zeta(s)$ increases as fast as the
Gamma function
when $\Re(s) <0$. It is therefore important to truncate the integral to $-T\leq \Im(s) \leq T$ for a small $T$. The   factor $1/s(s+1)$ helps
but is not   enough. One must   
take into account cancellation   in the values of $x^{\theta+it}$ as $t\to\infty$. This is the goal of the   next lemmas.\end{remark}
\begin{lemma}\label{eq:integral1} For all $
\theta \in [0,2],T\geq 2$ and $
y>0$ with $y\neq 1 $ we have 
$$ \int_{\substack{ \Re(s)=\theta \\ |\Im(s)|\geq T }}
\frac{y^s}{s(s+1)} \mathrm ds\ll \frac{y^\theta}{T^2|\log y|},$$ with an absolute implied constant.
\end{lemma}
\begin{proof} Writing $s=\theta+it $ we  use
$y^{it}=(y^{it})'/(i\log y)$
to get $\ll x^\theta(\log y)^{-1} (|I_1|+|I_2|)$, where $$ I_1=  \int_{T }^\infty 
\frac{(y^{it})'\mathrm dt}
{(\theta+it )(\theta+it +1)} =
-\frac{y^{iT}}{(\theta+iT)(\theta+iT+1)}
+\int_{T }^\infty 
\frac{i y^{it}(2\theta+2it+1)\mathrm dt}{(\theta+it)^2(\theta+it+1)^2}
,$$ and where $I_2$ is defined similarly with $t$ replaced by $-t$. Using $\theta=O(1)$
we see that \[|I_1|+|I_2| \ll \frac{1}{T^2}+
\int_{T }^\infty \frac{\mathrm dt}{t^3}\ll 
T^{-2}.\qedhere\]\end{proof}

\begin{lemma}\label{eq:integral2}
Let $f:\mathbb N\to \mathbb C$ be bounded by $1$ in modulus.
Then for all $x\in \mathbb N,\theta \in (1,2]$ we have 
$$\sum_{1\leq k \leq x } f(k)  (x-k)
=\frac{1}{2\pi i }\int_{\theta-iT}^{\theta+iT} 
\sum_{k=1}^\infty\frac{f(k)}{k^s}
\frac{x^{s+1} \mathrm ds }{s(s+1) }
+O\left(
\frac{x}{T}+ 
\frac{x^{1+\theta}}{T^2 (\theta-1)}
+\frac{x^2\log (x+1)}{T^2}  \right)
,$$ where the implied constant is absolute.
\end{lemma}
\begin{proof}By \cite[pages 142-143]{MR2378655}
for $y>0,\theta>1$ we have  
$$\frac{1}{2\pi i }\int_{\Re(s)=\theta} 
\frac{y^s }{s(s+1) } \mathrm ds
=\begin{cases}
1-1/y, &  y\geq  1  \\ 0, & 0<y\leq  1. \end{cases}
$$ Combining with Lemma \ref{eq:integral1}
we obtain the following   for all $T\geq 2 $ and $y>0$,
$$\frac{1}{2\pi i }\int_{\substack{ \Re(s)=\theta\\|t|\leq T }} 
\frac{y^s }{s(s+1) } \mathrm ds
=O\left(\frac{y^\theta}{T^2|\log y|}\right)+\begin{cases}
1-1/y, &  y> 1  \\ 0, & 0<y<  1 \end{cases}
$$ while, when $y=1$ the right-hand side is $\ll T^{-1}$. 
We obtain  $$\frac{1}{2\pi i }\int_{\theta-iT}^{\theta+iT} 
\frac{f(k)}{k^s}
\frac{x^s \mathrm ds }{s(s+1) }
+O\left ( 
\begin{cases} (x/k)^\theta T^{-2} , &  k\geq 1, k\neq x,   \\
T^{-1}, & k=x. \end{cases}
\right )=f(k) 
\begin{cases} 1-\frac{k}{x}, &  k\leq x,   \\
0, & k\geq x \end{cases}$$ 
Summing this over all $k$ gives 
$$\frac{1}{2\pi i }\int_{\theta-iT}^{\theta+iT} 
\sum_{k=1}^\infty  \frac{f(k)}{k^s}
\frac{x^{s+1}  }{s(s+1) } \mathrm ds
+O\left (
\frac{x}{T}  
+\frac{x^{1+\theta} }{T^2}
\sum_{k\neq x }^\infty
\frac{1}{k^\theta|\log(x/k)| }
 \right )=\sum_{1\leq k \leq x } f(k)  (x-k) .$$
The error term 
is 
$$\ll \frac{x}{T}+
\frac{x^{1+\theta}}{T^2} 
\sum_{k\neq x} \frac{ 1}{k^\theta |\log(x/k)| }
\ll
\frac{x}{T}+\frac{x^{1+\theta}}{T^2} 
\left(
\sum_{|k-x|>x/2 } \frac{ 1}{k^\theta   }
+\sum_{0<|k-x|\leq x/2} \frac{ 1}{ 
x^\theta|\log(x/k)| }\right)
.$$ We bound the sum over  $|k-x|>x/2$ by $\zeta(\theta)\ll 1/(\theta-1).$ 
For $k\in (x,3x/2]$ we have   
$$ \frac{1}{|\log (x/k)|}= 
\frac{1}{\log( 
1+\{k/x-1\})  
}\ll \frac{1}{k/x-1},$$ hence, the interval
$(x,3x/2]$ contributes 
$$\ll \frac{x}{T^2} 
\sum_{x<k\leq 3x/2} \frac{x}{k-x}
\ll 
\frac{x^2\log x}{T^2}  
.$$The interval $[x/2,x)$ is dealt with similarly.
\end{proof}
We now analyse the Dirichlet series of 
$$b(k):=\prod_{\substack{\ell \ \mathrm{ prime} \\ \ell=2}}^\infty
\left(1-\frac{2}{\ell^2} +
\frac{1}{\ell^2} 
\begin{cases}
\ell^{-2}, & \ell \nmid k, \\
\ell^{-1}, & v_\ell(k)=1, \\
1, & \ell^2 \mid k
\end{cases}
\right) 
.$$
\begin{lemma}\label{lem:localanalysis}
The  function  
$$G(s):=\prod_{\ell}
\left(1-\frac{1}{\ell^{2+2s}}\right) 
\left(1+ 
\frac{1}{(\ell-1)(\ell+1)^2} 
\frac{1}{\ell^s}+
\frac{\ell}{(\ell-1)(\ell+1)^2}
\frac{1}{\ell^{2s}}\right)
$$ is  holomorphic for $\Re(s)>-3/4$
and in that region it satisfies  $$\sum_{k=1}^\infty 
\frac{b(k) }{k^s}=\zeta(s)\zeta(2+2s) 
\frac{G(s)}{\zeta(2)^2}.$$  
\end{lemma}\begin{proof} We see that 
\begin{align*}
b(k)&= \prod_{\ell\nmid k}
\left(1-\frac{1}{\ell^2}\right)^2
\prod_{\ell\| k}
\left(1-\frac{2}{\ell^2}+\frac{1}{\ell^3}\right)
\prod_{\ell^2\mid k}
\left(1-\frac{1}{\ell^2}\right)\\ 
&= \frac{1}{\zeta(2)^2}
\prod_{\ell\| k}
\left(1+ \frac{1}{(\ell-1)(\ell+1)^2}\right)
\prod_{\ell^2\mid k}
\left(1+\frac{1}{\ell^2-1}\right)  
=\frac{(1\ast g)(k)}{\zeta(2)^2},\end{align*} 
where $\ast$ is the Dirichlet convolution and 
for a prime $\ell$ we have $$g(\ell)=
\frac{1}{(\ell-1)(\ell+1)^2},
\ \ \ g(\ell^2)=\frac{\ell}{(\ell-1)(\ell+1)^2}, 
\ \ \ g(\ell^q)=0 \ \ \forall q\geq 3.$$
Hence, for $\Re(s)>1$ we have  
$$\sum_{k=1}^\infty 
\frac{b(k) }{k^s}=\frac{\zeta(s)}{\zeta(2)^2}
\prod_{\ell}\left(1+ \frac{1}{(\ell-1)(\ell+1)^2 \ell^s}
+\frac{\ell}{(\ell-1)(\ell+1)^2 \ell^{2s}} 
\right) .$$ 
One can verify that the right-hand 
side equals $ \frac{\zeta(s)}{\zeta(2)^2}
\zeta(2+2s) G(s)$, where $G$ is defined in 
the statement of the lemma.
Writing $\sigma=\Re(s), t=\Im(s)$
and fixing a prime $\ell$, 
we expand out the product of the two $\ell$-adic factors in the definition of $G$. 
We obtain an expression in which 
$\ell^{-2-2s}$ cancels out and we are left
with an expression of type $1+
s_1\ell^{-f_1}+\ldots+s_n \ell^{-f_n}$, where 
$n$ is finite and independent of $\ell$,
the constants 
$s_i$ are signs in $\{1,-1\}$ and $f_i$ are   functions of $s$ that satisfy 
$|f_i|>1$ when $\sigma>-3/4$. In particular, 
 $G$ is holomorphic in the region $\Re(s)>-3/4$.  \end{proof}
Using $b(m-n)=c(m,n)$, 
Lemma \ref{eq:integral2}  with $f=b$ and $\theta=1+1/\log x$
and Lemma \ref{lem:localanalysis}
we get  \begin{equation}
\label{eq:perron bira}
\sum_{1\leq n\neq m\leq x } c(m,n)
=\frac{1}{\pi i \zeta(2)^2 }
\int_{1+1/\log x-iT}^{1+1/\log x+iT} 
\zeta(s)\zeta(2+2s) 
G(s)\frac{x^{s+1} \mathrm ds }{s(s+1) }
+O\left(\frac{x^{2} \log x}{T^2 } \right)
\end{equation}
for   $T\leq x$  and  $x\geq 2$.
The next result gives an asymptotic expansion for the expression above. We remark that  the calculation for $\gamma_1$ will involve $\zeta(0)=-1/2$,  the minus sign  of which 
cancels   the diagonal contribution, as mentioned in
the introduction.
\begin{lemma}\label{lem:perronapplication} 
For any $\delta\in(0,1/4)$,
any integer $x\geq 2$ and $T\leq x$ we have
$$
\sum_{1\leq n\neq m\leq x } c(m,n)
=\gamma_2 x^2 +\gamma_1 x +\gamma_0 x^{1/2} 
+O\left(|J_1|+|J_2|+|J_3|+\frac{x^{2} \log x}{T^2 } \right)$$ with an absolute implied constant,
where $\gamma_0$ is as in \eqref{Leo_Artemisia}
and $$\gamma_2 :=   
\prod_{\substack{\ell \ \mathrm{ prime} \\ \ell=2}}^\infty
\left(1+ \frac{2}{(\ell-1)(\ell+1)^2 \ell}
\right)
\ \ \ \textrm{ and } \ \ \ 
\gamma_1:=-\frac{1}{\zeta(2)}.$$ Here $J_i$ are   three  
integrals as in the right-hand side of 
\eqref{eq:perron bira} 
resulting by replacing the range of integration respectively by 
\begin{align*}
&\left[-\frac{1}{2}-\delta+iT,1+\frac{1}{\log x}+iT\right]
,
\left[-\frac{1}{2}-\delta-iT,-\frac{1}{2}-\delta+iT\right]
,
\left[-\frac{1}{2}-\delta-iT,1+\frac{1}{\log x}-iT\right]
.\end{align*}
\end{lemma}\begin{proof} 
By Cauchy's theorem the integral in the 
right-hand side of \eqref{eq:perron bira}
equals $$ 
\sum_{j\in \{1,0,-1/2\}}
\mathrm{Res}\left(\zeta(s)\zeta(2+2s) 
G(s)\frac{x^{s+1}  }{s(s+1) };s=j\right)
+J_1+J_2-J_3$$ since $\delta<1/4$ is small enough 
so that the pole at $s=-1$ is outside the box.
Here, $\mathrm{Res}(F;s=j)$ stands for the residue of $F$ at $s=j$. 
When $j=1$ 
the residue is $$\zeta(4) G(1)\frac{x^2  }{2
  }
=  \frac{x^2  }{2 }\zeta(2)^2
 \prod_{\substack{\ell \ \mathrm{ prime} \\ \ell=2}}^\infty\left( 1-\frac{2}{\ell^2}+
\frac{3-2/\ell}{\ell^3}
\right),$$by     the definition of $G$ in Lemma \ref{lem:localanalysis}. To obtain $\gamma_2$ this must be multiplied by
$2/\zeta(2)^2$ due to \eqref{lem:perronapplication}.
When $j=0$ the residue is 
$\zeta(0)\zeta(2) G(0)x =
-\zeta(2) x/2,$ since each $\ell$-adic factor in the definition of $G(0)$ in Lemma \ref{lem:localanalysis} equals $1$. Again, to obtain $\gamma_1$ we must multiply the residue by  
$2/\zeta(2)^2$. Finally, the definition of $G$ in   Lemma \ref{lem:localanalysis} 
gives \[ G(-1/2)=\prod_{\ell} 
\left(1-\frac
{(2  - \ell^{-1/2}+ \ell^{-1})}
{(\ell + 1)^2 }
\right).\] Hence, 
when $j=-1/2$ the residue is $$\zeta(-1/2)
\frac{1}{2}
G(-1/2)\frac{x^{1/2}  }{(-1/2)(1/2)}
=-2\zeta(-1/2)G(-1/2) x^{1/2}$$ and  one must multiply this $2/\zeta(2)^2$ to obtain $\gamma_0$. \end{proof}

\section{Perron integral}   
\label{s:peron}
\subsection{Zeta prerequisites}  
We recall some estimates
regarding the Riemann zeta function. The usual 
notation $\sigma=\Re(s), t=\Im(s)$ is used throughout this section.

The functional equation and bounds for the 
  Gamma function are used in~\cite[Equation (5.1.1)]{MR882550}  to show that  the following bound   holds
with an absolute implied constant: 
   \begin{equation}\label{eq:bnd1} 
\zeta(s)\ll |t|^{1/2-\sigma }  , \ \ \ 
-1\leq \sigma \leq -1/10. \end{equation}
When $\sigma$ is around $1$,
~\cite[Theorem 3.5]{MR882550}
states   the bound
   \begin{equation}\label{eq:bnd3} 
\zeta(s)\ll \log |t| , \ \ \ 
1-\frac{2}{\log |t|}\leq \sigma\leq 2,
|t|\geq 1 \end{equation}
with an absolute implied constant.

The functional equation reads 
$\zeta(s)=\chi(s) \zeta(1-s)$, where 
\[\chi(s)=2^s \pi^{s-1}\Gamma(1-s)
\sin\left(\frac{\pi s}{2}\right) .\] 
The bound $|\chi(s)| \ll |t|^{1/2 -\sigma}$   
holds uniformly  for   $\sigma \in (-1,1)$
by~\cite[page 95]{MR882550}.
Thus, by \eqref{eq:bnd3},\begin{equation}\label{eq:bnd7} 
   \zeta(s)\ll  |t|^{1/2-\sigma } \log |t|, \ \ \ 
 -1<\sigma\leq \frac{1}{\log |t|}.\end{equation}
  When $\sigma \in (0,1)$ 
the  case of the approximate functional 
equation~\cite[Theorem 4.13]{MR882550} 
corresponding to the choice of parameters 
$ x=y=\sqrt{|t|/(2\pi)}$  states that 
$\zeta(s)$ equals 
$$\sum_{n\leq x} n^{-s}+
\chi(s) \sum_{n\leq x} \frac{1}{n^{1-s}}
+O(x^{-\sigma} (\log |t|) + 
|t|^{1/2-\sigma} x^{\sigma-1})
\ll  \frac{|t|^{(1-\sigma)/2}}{1-\sigma}  
+ |\chi(s)| \frac{|t|^{\sigma/2}}{\sigma}
+ \frac{\log |t|}{ |t|^{\sigma/2} } 
.$$
Combining this with
$|\chi(s)| \ll |t|^{1/2 -\sigma}$ we get 
\begin{equation}\label{pergolesi:salustia}
\zeta(s)\ll \frac{ |t|^{\frac{1-\sigma }{2} } }
{\min\{\sigma,1-\sigma\} }+\log |t|, \ \ \ \sigma \in (0,1). \end{equation} 
The following mean value 
theorem is given in~\cite[Theorem 7.2]{MR882550}
with an absolute implied constant:
\begin{equation}\label{eq:bnd2} 
\int_1^T |\zeta(\sigma+it)|^2\mathrm dt
\ll T\min\left\{\log T, \frac{1}{\sigma-1/2}\right\} , \ \ \ 
1/2\leq \sigma \leq 2. \end{equation}

From now on we assume that\begin{equation}\label{eq:mvesp}x^{3/4}\leq T\leq x. 
\end{equation}
This ensures that the term $x^2/T^2$ in Lemma \ref{lem:perronapplication}  is $\leq x^{1/2}$.

\subsection{Vertical integral}
Writing $s=-1/2-\delta+it$ we can use
 \eqref{eq:bnd1}   to get 
$$J_2\ll \int_{-T}^T 
t^{1+\delta  } |\zeta(1-2\delta+ it)|
\frac{x^{1/2-\delta}\mathrm dt}{1+t^2} 
\ll
x^{1/2}(T/x)^\delta 
\int_1^T     \frac{|\zeta(1-2\delta+ it)|}{t }
\mathrm dt.$$  Since $\delta<1/4$ we have 
$1-2\delta>1/2$, hence, \eqref{eq:bnd2} 
gives 
$\int_z^{2z}|\zeta(1-2\delta+ it)|^2 
\mathrm dt \ll z$, with an   implied 
constant depending only on $\delta$. Therefore, 
\begin{align*}
\int_1^T \frac{|\zeta(1-2\delta+ it )| }{t}\mathrm d t
&\leq 
\sum_{0\leq j \leq \log T}\mathrm e^{-j}
\int_{\mathrm e^j}^
{\mathrm e^{j+1}}
|\zeta(1-2\delta+ it )| \mathrm d t
\\
&\ll  
\sum_{0\leq j \leq \log T}\mathrm e^{-j/2}
\left(\int_{\mathrm e^j}^
{\mathrm e^{j+1}}
|\zeta(1-2\delta+ it )|^2 \mathrm d t
\right)^{1/2}
\ll  \log T.
 \end{align*} Thus, for all 
 $\delta \in (0,1/4)$  
 the following holds with an absolute implied 
 constant: 
 \begin{equation}
 \label{eq:JJJ222}J_2\ll x^{1/2}(T/x)^\delta \log T.
 \end{equation}

\subsection{Horizontal integral: near $1$} 
We   focus on   $J_1$ since 
$J_3$  is   treated similarly.
We have    
$$J_1=  \int_{-\frac{1}{2}-\delta}^{1+\frac{1}{\log x}} 
 \zeta(\sigma+ iT )
\zeta(2\sigma+2+ iT )
\frac{G(\sigma+ iT )x^{1+\sigma+ iT}\mathrm d \sigma}{(\sigma+ iT )(\sigma+1+ iT )}.$$ 
The two different zeta factors 
contribute in different regions and it is thus necessary to consider many different intervals.
We write  \begin{align*}
\begin{alignedat}{4}
I_1 &:= \left[ 1-\frac{1}{\log x},\, 1+\frac{1}{\log x} \right],
&\qquad
I_2 &:= \left[ \frac{1}{\log T},\, 1-\frac{1}{\log T} \right], \\
I_3 &:= \left[ -\frac{1}{2}-\frac{1}{\log T},\, \frac{1}{\log T} \right],
&\qquad
I_4 &:= \left[ \frac{1}{2}-\delta,\, \frac{1}{2}-\frac{1}{\log T} \right] \end{alignedat}
\end{align*} and denote the corresponding 
subintegral of $J_1$ by $J_{1k}$ for $k=1,\ldots,4$.

By \eqref{eq:bnd3} we have    \begin{equation}
 \label{eq:mjikna8} J_{11} \ll 
\int_{ 1-1/\log T}^{1-1/\log x}
(\log T) \frac{x^{1+\sigma}
\mathrm d \sigma}{T^2}
\ll   \frac{x^2}{T^2}
. \end{equation}

\subsection{Horizontal integral: positive}\label{s:hrzntl}
By \eqref{pergolesi:salustia} we obtain 
$$ J_{12}\ll  (\log T)
T^{1/2-2} x
\int_{1/\log T}^{1-\frac{1}{\log T}}  
T^{\frac{-\sigma }{2} }
x^{\sigma}\mathrm d \sigma
\ll \frac{ x(\log T) }{T^{3/2}} 
\int_{1/\log T}^{1-1/\log T} 
(x  T^{-1/2} )^\sigma 
 \mathrm d \sigma.$$
 Since $T\leq x$
the integral is $\ll (x  T^{-1/2} )^{1-1/\log x}
 \ll x  T^{-1/2} $. Hence,  
   \begin{equation}
 \label{eq:porpora amato nome} J_{12}   
 \ll \frac{ x^2 }{T^{2}} \log T. \end{equation}   
\subsection{Horizontal integral: negative part}
By \eqref{eq:bnd7} we obtain  
 $$ J_{13}\ll 
  \frac{x \log T}{T^{3/2} } 
\int_{-\frac{1}{2}-\frac{1}{\log T}}^{\frac{1}{\log T}}   
(x/T)^{ \sigma }  |\zeta(2\sigma+2+ 2iT ) |
\mathrm d \sigma.$$ The estimate 
$|\zeta(2\sigma+2+ 2iT ) |\ll \log T$
follows from \eqref{eq:bnd3}. It yields the bound 
\begin{equation}\label{eq:jj1313}
J_{13}\ll 
  \frac{x (\log T)^2}{T^{3/2} } 
  \int_{-\frac{1}{2}-\frac{1}{\log T}}^{\frac{1}{\log T}} 
(x/T)^{  \sigma }   \mathrm d \sigma
  \ll
  \frac{  x ^{1+1/\log T}}{T^{3/2} }   (\log T)^2 
 \ll
  \frac{  x }{T^{3/2} }   (\log T)^2   \ll
  \frac{  x^2 }{T^{2} }   (\log T)^2 
  ,\end{equation}
  where the two last bounds follow from \eqref{eq:mvesp}. 
 \subsection{Horizontal integral: near $-1/2$}
 \label{s:fnlaly}
By \eqref{eq:bnd7} we get   
 $$ J_{14}\ll 
  \frac{x \log T}{T^{3/2} } 
  \int_{-\frac{1}{2}-\delta}^
  {-\frac{1}{2}-\frac{1}{\log T}} 
(x/T)^{ \sigma }  |\zeta(2\sigma+2+ 2iT ) |
\mathrm d \sigma.$$  Employing~\eqref{pergolesi:salustia} with $s$ replaced by $2s+2$ yields the  bound 
  $$ \zeta(2\sigma+2+ 2iT )\ll 
T^{\frac{-1-2\sigma }{2}} \log T$$ uniformly in the range of the integral. Hence,  by \eqref{eq:mvesp} we get 
\begin{equation}\label{eq:j1414}
J_{14}\ll 
  \frac{x (\log T)^2}{T^2} 
  \int_{-\frac{1}{2}-\delta }
  ^{-\frac{1}{2}-\frac{1}{\log T}} 
   (x/T^2)^{ \sigma } 
\mathrm d \sigma 
\ll   \frac{x^{1/2}  (\log T)^2}{T}
\left (  \frac{T^2}{x}\right)^{\delta}
\ll   x^{1/2}(T/x)^\delta (\log T)^2 
,\end{equation} where the last bound is from 
$\delta<1$.
Putting together 
\eqref{eq:JJJ222}-\eqref{eq:j1414} 
shows that the   error term in 
Lemma \ref{lem:perronapplication}
is 
$$\ll 
\left(
\frac{x^{2} }{T^2 } + x^{1/2}(T/x)^\delta
\right)(\log x)^2
.$$ 
Solving $x^2T^{-2}=x^{1/2}(T/x)^\delta$ yields
$T=x^{\omega}$ with $\omega=(\delta+3/2)/(\delta+2)$.
Substituing back we get the error term 
$$\ll  
x^{2-\omega} (\log x)^2=x^{1/(\delta+2)} (\log x)^2
.$$ Take $\delta=1/4-\epsilon'$ for a fixed   $\epsilon'\in (0,1/4)$  so that, 
for any  fixed $\epsilon>0$, the above becomes 
$O_\epsilon(x^{\epsilon+4/9})$.
Hence, by Lemma \ref{lem:perronapplication} 
we infer that for each fixed $\epsilon>0$ 
and any integer $x\geq 2$ one has 
\begin{equation}\label{eq:brokenglas}
\sum_{1\leq n\neq m\leq x } c(m,n)
=\gamma_2 x^2 +\gamma_1 x +\gamma_0 x^{1/2} 
+O_\epsilon(x^{4/9+\epsilon} ).
\end{equation} 
\begin{remark}\label{rem:nolindelof} 
Let us see why the Lindel\"of hypothesis, or bounds towards it,
do not help. These   can be used to 
bound  $\zeta(s)$ in  $J_{12}$
or   $\zeta(2s+2)$ in $J_{14}$.
The convexity bound \eqref{pergolesi:salustia}
used in \S\ref{s:hrzntl}   produces the bound
\eqref{eq:jj1313} that is not worse than the one   already present 
 in Lemma \ref{lem:perronapplication}. Assuming a subconexity bound, i.e. that 
 there is $\beta\in [0,1/4)$ such that 
 $\zeta(1/2+it)=O(|t|^{\beta+\epsilon})$
for $\epsilon>0$ and $|t|\geq 1$
implies by the Phragm\'en--Lindel\"of principle
that  
$\zeta(\sigma+it)=O(|t|^{2\beta(1-\sigma)+\epsilon})$
uniformly in the range $|t|\geq 1$ and $1/2\leq \sigma \leq 1$. 
Using this bound instead of~\eqref{pergolesi:salustia} in \S\ref{s:fnlaly} leads to 
$$ J_{14}\ll x^{1/2}(T/x)^\delta (T^{4\beta\delta-1}) x^\epsilon.$$
Save for  the harmless term $x^\epsilon$, this bound is not worse 
than the one already present in \eqref{eq:JJJ222}
due to   $\beta<1/4$ and $\delta\leq 1/4$.
\end{remark}

\section{Conclusion of the proof} 
Summing the asymptotic of 
Lemma~\ref{lem:morte d'abel avondano}
over all $1\leq n\neq m\leq x$ and using~\eqref{eq:brokenglas} 
to estimate the main term  leads to 
\begin{align*}
\sum_{n\neq m \leq x}
\sum_{P\in \c P_d(H)}
\mu^2_{\mathrm{model}}(P(n))
\mu^2_{\mathrm{model}}(P(m))
&=(2H)^{d+1} \left(
\gamma_2 x^2 +\gamma_1 x +\gamma_0 x^{1/2} +O_\epsilon(x^{\epsilon+4/9})\right)\\
&+ O\left(x^2
z^{4d+2}+x^2z^2H^d+
\frac{x^2H^{d+1} \log z}{z}
\right). \end{align*} Noting that $\gamma=\gamma_2$
and alluding to     
Lemmas~\ref{lem:cutdown} and~\ref{lem:middle}
we get
$$V(H,x,z)=(2H)^{d+1}  
 \gamma_0 x^{1/2} +O(H^\epsilon\c R),$$ 
 with $
\c R= H^{d+1} x^{ 4/9} +x^{2+d/2} H^{d+1/2} 
+x^2 H^{d+1}/z
+x^2z^{4d+2}
 +x^2z^2H^d  
.$  We take $z=H^\lambda$ with $\lambda=(d+1)/(4d+3)$.
This choice   is   optimal   
for minimising the last three terms.
We obtain  $\c R\ll  H^{d+1} x^{ 4/9} 
+x^{2+d/2} H^{d+1/2} 
+x^2 H^{d+1-\lambda}
$. We infer   that for $d\geq 3$ one has 
$$
x   \leq H^{\frac{1}{d+3}}
\ \ \ \Rightarrow \ \ \ 
\c R \ll H^{d+1}  x^{1/2} 
\ \ \textrm{ and } \ \ 
x \leq H^{ \frac{1}{d+28/9} }  
\ \ \Rightarrow \ \ 
\c R \ll H^{d+1} x^{4/9} 
.$$ These respectively prove the first and second part of 
Theorem \ref{mmm:ttt} 
by using
Lemma \ref{lem:initialsplit} 
and noting that the error term is 
$$
\ll V(H,x,z)^{1/2} x
\left( \frac{H^{d+1+\epsilon}}{z^{2-\epsilon}} \right)^{1/2}
+  x ^2  \frac{H^{d+1+\epsilon}}{z^{2-\epsilon}} 
\ll x^{5/4} H^{d+1+2\epsilon} z^{-1} 
+  x ^2 H^{d+1+2\epsilon} z^{-1}
,$$ which is   admissible.

\end{document}